\documentclass[12pt, reqno]{amsart}
\usepackage[utf8]{inputenc}
\usepackage{indentfirst}
\usepackage{amsmath, amssymb, amsthm, mathrsfs}
\usepackage{setspace}
\usepackage{wasysym}
\usepackage{enumerate}
\usepackage[colorlinks=true, linkcolor=purple, citecolor=blue, urlcolor=magenta]{hyperref}
\usepackage[table]{xcolor}
\usepackage{graphicx}
\usepackage{tikz}
\usepackage{float}
\usepackage{booktabs}
\newtheorem{theo}{Theorem}[section]
\newtheorem{lem}{Lemma}[section]

\numberwithin{equation}{section}

\begin{document}

\title[Higher-Order Schippers-Schwarzian Derivatives]{Higher-Order Schippers-Schwarzian Derivatives for Non-Circular Starlike Functions}

\author[P. Das, P. Zaprawa and N. Sarkar]{Pradip Das, Pawe\l{} Zaprawa and Nabadwip Sarkar}

\address{Department of Mathematics, Raiganj University, Raiganj, West Bengal-733134, India.}
\email{pradipsmath@gmail.com}

\address{Lublin University of Technology, Nadbystrzycka 38D, Lublin 20-618, Poland.}
\email{p.zaprawa@pollub.pl}

\address{Amity School of Applied Sciences, Amity University Mumbai, Panvel, Navi Mumbai, Maharashtra-410206, India.}
\email{nsarkar@mum.amity.edu}

\renewcommand{\thefootnote}{}
\footnote{*\emph{Corresponding Author}: Pawe\l{} Zaprawa.}
\renewcommand{\thefootnote}{\arabic{footnote}}
\setcounter{footnote}{0}

\begin{abstract}
We find sharp upper bounds for the initial higher-order Schippers-Schwarzian derivatives $\sigma_3(f)(0)$ and $\sigma_4(f)(0)$ for several subclasses of univalent and starlike functions in the unit disk. The functions in these classes are defined by subordination to domains bounded by an exponential-sine curve, a cardioid, or a petal-shaped curve. We determine the sharp bounds and find the corresponding extremal functions for each invariant in these subclasses. As an application, we also obtain sharp bounds for the initial Grunsky coefficients $\omega_{1,1}$ and $\omega_{1,2}$.
\end{abstract}

\maketitle

\section{Introduction}

For a locally univalent function $f$ defined on the unit disk $\mathbb{D} = \{z \in \mathbb{C} : |z| < 1\}$, the classical Schwarzian derivative $S_f$ is defined as
\[
S_f = \left(\frac{f''}{f'}\right)' - \frac{1}{2}\left(\frac{f''}{f'}\right)^2.
\]
This operator is invariant under M\"{o}bius transformations. Specifically, the identity $S_{L\circ f} = S_f$ holds for any linear fractional transformation $L(z) = \frac{az+b}{cz+d}$ with $ad - bc \neq 0$. Moreover, $S_f = 0$ if and only if $f$ is a M\"{o}bius transformation. A theorem of Nehari \cite{Nehari} states that if $f$ is univalent in $\mathbb{D}$, then
\[
\left\|S_f\right\| = \sup_{z \in \mathbb{D}} |S_f(z)|(1 - |z|^2)^2 \le 6.
\]
Bounds of this type have been studied for various subclasses of univalent functions (see \cite{KanasSugawa, Schippers}).

Higher-order generalizations of the Schwarzian derivative have been introduced by several authors (see \cite{Harmelin, HuSrivastavaZhang, Tamanoi}). In particular, Schippers \cite{Schippers} defined the higher-order Schwarzian derivatives $\sigma_n(f)$ recursively for $n \ge 3$ by
\[
\sigma_{n+1}(f) = \sigma_n'(f) - (n-1)\sigma_n(f)\frac{f''}{f'},
\]
where $\sigma_3(f) = S_f$. These operators are related to the Grunsky coefficients, which provide necessary and sufficient conditions for univalence.

Let $\mathcal{H}$ denote the class of analytic functions in $\mathbb{D}$, and let $\mathcal{A} = \{f \in \mathcal{H} : f(0) = 0, \, f'(0) = 1\}$. Let $\mathcal{S} \subset \mathcal{A}$ be the class of univalent functions. Each $f \in \mathcal{S}$ can be written as
\begin{equation}\label{eq1}
f(z) = z + \sum_{n=2}^{\infty} a_n z^n.
\end{equation}
By direct computation using \eqref{eq1}, we have
\[
\frac{f''(z)}{f'(z)} = 2a_2 + (6a_3 - 4a_2^2)z + (12a_4 - 18a_2 a_3 + 8a_2^3)z^2 + \cdots.
\]
The values of the initial higher-order Schwarzian derivatives at the origin are given by
\begin{equation}\label{sg3}
\sigma_3(f)(0) = 6(a_3 - a_2^2),
\end{equation}
\begin{equation}\label{sg4}
\sigma_4(f)(0) = 24(a_4 - 3a_2 a_3 + 2a_2^3).
\end{equation}

The values $|\sigma_n(f)(0)|$ for $n=3,4$ have been determined for several subclasses of univalent functions. For instance, Dorff and Szynal \cite{DorffSzynal} found these bounds for the class $\mathcal{C}$ of convex functions. Cho et al. \cite{ChoKumarRavichandran} extended these results to subclasses defined by Janowski subordination, namely
\[
\frac{z f'(z)}{f(z)} \prec \frac{1 + A z}{1 + B z} \quad \text{and} \quad 1 + \frac{z f''(z)}{f'(z)} \prec \frac{1 + A z}{1 + B z},
\]
where $-1 \le B < A \le 1$. We recall that $f$ is starlike ($\mathcal{S}^*$) or convex ($\mathcal{C}$) if $\Re(z f'(z)/f(z)) > 0$ or $\Re(1 + z f''(z)/f'(z)) > 0$ for $z \in \mathbb{D}$, respectively. Bounds for other subclasses have been obtained by Kumar et al. \cite{KumarChoRavichandranSrivastava} and Hu et al. \cite{HuWangFan}.

For two analytic functions $f$ and $g$ in $\mathbb{D}$, $f$ is subordinate to $g$, written $f \prec g$, if there exists an analytic function $\omega$ in $\mathbb{D}$ with $\omega(0)=0$ and $|\omega(z)|<1$ such that $f(z) = g(\omega(z))$.

This paper extends these results to subclasses associated with cardioid, petal-shaped, and exponential-sine subordinant domains. We determine sharp upper bounds for $|\sigma_3(f)(0)|$ and $|\sigma_4(f)(0)|$ for the following three subclasses of functions:
\begin{enumerate}
\item \textbf{The class $\mathcal{P}^{*}$:} Defined by the condition \cite{TAH}:
\[
f'(z) \prec e^z(1+\sin z).
\]
\item \textbf{The class $\mathcal{S}^{*}_{\mathrm{car}}$:} Defined by the condition \cite{SS}:
\[
\frac{z f'(z)}{f(z)} \prec 1 + \frac{4}{3}z + \frac{2}{3}z^2.
\]
\item \textbf{The class $\mathcal{S}^*_{\rho}$:} Defined by the condition \cite{Arora-Kumar-2022}:
\[
\frac{z f'(z)}{f(z)} \prec 1 + \sinh^{-1}(z).
\]
\end{enumerate}

The geometric characteristics of these subclasses are determined by the image domains of their corresponding subordinating functions. For the convenience of the reader, the associated image domains are illustrated in Figure~\ref{fig:domains}.

\begin{figure}[htbp]
    \centering
    \includegraphics[width=\textwidth]{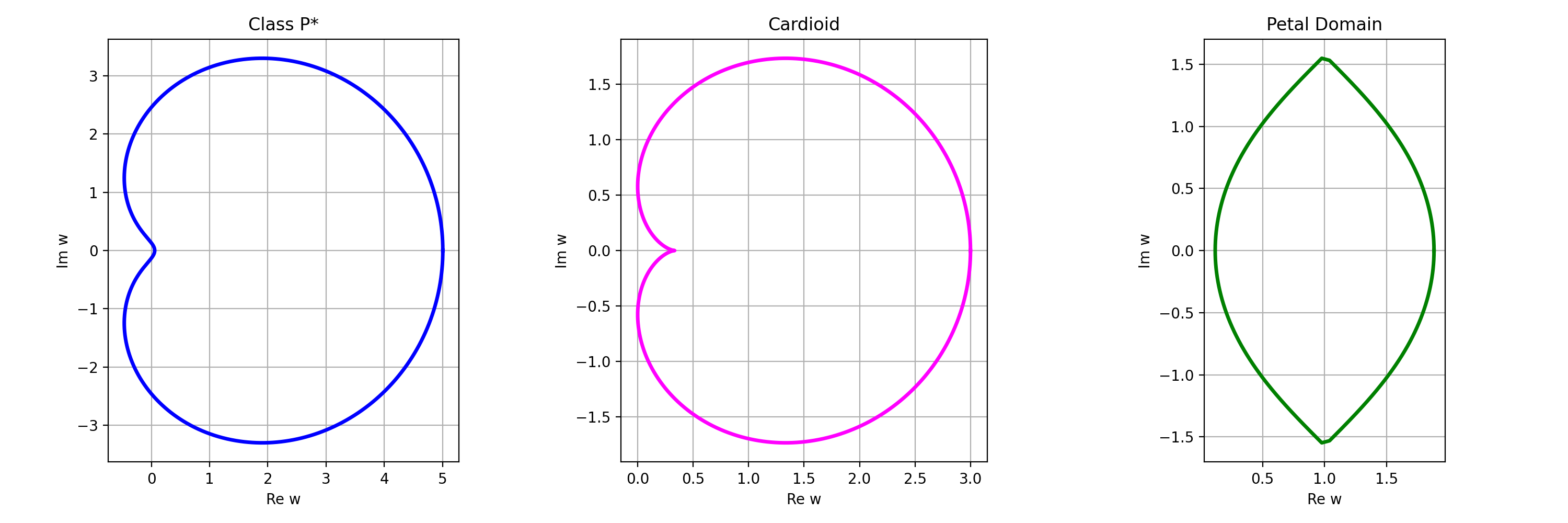}
    \caption{Image domains of the subordinating functions corresponding to (a) the class $P^{*}$, (b) the cardioid domain, and (c) the petal domain.}
    \label{fig:domains}
\end{figure}
The main aim of this paper is to find sharp upper bounds for $|\sigma_3(f)(0)|$ and $|\sigma_4(f)(0)|$ for these three classes, and to determine the extremal functions. By using a coefficient lemma of Prokhorov and Szynal \cite{Prokhorov}, we simplify the coefficient calculations and obtain the extremal functions. As an application, we also derive sharp bounds for the initial Grunsky coefficients $\omega_{1,1}$ and $\omega_{1,2}$.

The geometric behavior of a univalent function depends strongly on its derivative operators. The bounds obtained in this paper have important geometric meaning: 

The Schwarzian derivative $\sigma_3(f)$ measures how much a conformal map differs from a M\"{o}bius transformation. The higher-order derivatives $\sigma_4(f)$ describe more detailed local changes. Our sharp bounds show how cardioid, petal, and exponential-sine domains control twisting and stretching near the center of the unit disk.

Subordination to cardioid or petal-shaped domains restricts expressions such as $zf'(z)/f(z)$. Geometrically, this controls the shape of the boundary and prevents sharp distortions. The bounds on $\sigma_n(f)(0)$ give limits on the local curvature. This shows that non-circular domains restrict function growth more strongly than circular domains. 

The Schippers--Schwarzian operators are closely related to the first Grunsky coefficients $(\omega_{1,1}, \omega_{1,2})$. Our results give sharp bounds connected with the classical Area Theorem. They explain how area is distributed among the initial Taylor coefficients and describe stability through the invariant quantity $\mathcal{E}_4(f)$.

\section{Auxiliary lemmas}

Let $\Omega$ denote the class of Schwarz functions $\omega$ analytic in $\mathbb{D}$ satisfying $\omega(0)=0$ and $|\omega(z)| < 1$, with the Taylor series representation:
\begin{equation}\label{zp1}
\omega(z) = \sum_{n=1}^{\infty} c_n z^n = c_1 z + c_2 z^2 + c_3 z^3 + \dots, \quad z \in \mathbb{D}.
\end{equation}

The following lemmas are required to establish our results. First, we state a result by Ma and Minda \cite{MM} for a second-order coefficient functional over the Carath\'{e}odory class $\mathcal{P}$.

\begin{lem}\label{lemma1}
Let $p \in \mathcal{P}$ be given by $p(z) = 1 + \sum_{n=1}^{\infty} b_n z^n$ with $\Re p(z) > 0$ in $\mathbb{D}$. Then
\[
\left| b_2 - v b_1^2 \right| \le 
\begin{cases}
-4v + 2, & v < 0, \\
2, & 0 \le v \le 1, \\
4v - 2, & v > 1.
\end{cases}
\]
For $v < 0$ or $v > 1$, equality holds if and only if $h(z) = \frac{1+z}{1-z}$ or one of its rotations. For $0 < v < 1$, equality holds if and only if $h(z) = \frac{1+z^2}{1-z^2}$ or one of its rotations.
\end{lem}

To evaluate the fourth-order expressions, we use a lemma by Prokhorov and Szynal \cite{Prokhorov} for a linear combination of the first three coefficients of a Schwarz function.

\begin{lem}\label{lemma2}\cite{Prokhorov}
Let $\omega \in \Omega$ be given by \eqref{zp1}. Then, for any real numbers $\mu$ and $\nu$,
\[
\left| c_3 + \mu c_1 c_2 + \nu c_1^3 \right| \le \Phi(\mu, \nu),
\]
where $\Phi(\mu, \nu)$ is defined on different regions of the $(\mu, \nu)$ plane by:
\[
\Phi(\mu, \nu) = 
\begin{cases}
1, & \text{if } (\mu,\nu) \in D_1 \cup D_2 \cup \{(2,1)\}, \\
|\nu|, & \text{if } (\mu,\nu) \in \bigcup_{k=3}^{7}D_k, \\
\frac{2}{3}(|\mu|+1)\left( \frac{|\mu|+1}{3(|\mu|+1+\nu)} \right)^{1/2}, & \text{if } (\mu,\nu) \in D_8 \cup D_9, \\
\frac{1}{3}\left( \frac{\mu^2-4}{\nu} \right) \left( \frac{4-\mu^2}{3(1-\nu)} \right)^{1/2}, & \text{if } (\mu,\nu) \in D_{10} \cup D_{11} \setminus \{(2,1)\}, \\
\frac{2}{3}(|\mu|-1)\left( \frac{|\mu|-1}{3(|\mu|-1-\nu)} \right)^{1/2}, & \text{if } (\mu,\nu) \in D_{12}.
\end{cases}
\]
The regions relevant to our proofs are:
\begin{align}
D_5 &= \{(\mu,\nu) : |\mu| < 2, \, \nu \ge 1 \}, \\
D_6 &= \left\{(\mu,\nu) : 2 \le |\mu| < 4, \, \nu \ge \frac{1}{12}(\mu^2+8) \right\}, \\
D_7 &= \left\{(\mu,\nu) : |\mu| \ge 4, \, \nu \ge \frac{2}{3}(|\mu|-1) \right\}.
\end{align}
\end{lem}

\begin{lem}\label{lemma3}
Let $p \in \mathcal{P}$ be of the form $p(z) = 1 + \sum_{n=1}^{\infty} b_n z^n$. If $\omega \in \Omega$ is defined by $\omega(z) = \frac{p(z)-1}{p(z)+1} = \sum_{n=1}^{\infty} c_n z^n$, then:
\begin{align}
c_1 &= \frac{1}{2}b_1, \label{zp2}\\
c_2 &= \frac{1}{2}b_2 - \frac{1}{4}b_1^2, \label{zp3}\\
c_3 &= \frac{1}{2}b_3 - \frac{1}{2}b_1 b_2 + \frac{1}{8}b_1^3, \label{zp4}\\
c_4 &= \frac{1}{2}b_4 - \frac{1}{2}b_1 b_3 - \frac{1}{4}b_2^2 + \frac{3}{8}b_1^2 b_2 - \frac{1}{16}b_1^4. \label{zp5}
\end{align}
\end{lem}
\section{Main Results and Proofs}

\subsection{Higher-Order Schippers-Schwarzian Derivatives for the Class $\mathcal{P}^*$}
In this section, we determine the sharp upper bounds for the initial functionals $|\sigma_3(f)(0)|$ and $|\sigma_4(f)(0)|$ for functions belonging to the class $\mathcal{P}^*$.

\begin{theo}\label{T1}
Let $f \in \mathcal{P}^*$ be given by $f(z) = z + \sum_{n=2}^{\infty} a_n z^n$. Then
\[
|\sigma_3(f)(0)| \le 4 \quad \text{and} \quad |\sigma_4(f)(0)| \le 15.
\]
These bounds are sharp.
\end{theo}

\begin{proof}
Let $f \in \mathcal{P}^*$. By definition, there exists a Schwarz function $\omega \in \Omega$ with $\omega(z) = c_1 z + c_2 z^2 + c_3 z^3 + \dots$ such that
\begin{equation}\label{pd1}
f'(z) = e^{\omega(z)}\big(1 + \sin \omega(z)\big).
\end{equation}
Expanding the series on the right-hand side of \eqref{pd1} and comparing coefficients with the standard formula $f'(z) = 1 + 2a_2 z + 3a_3 z^2 + 4a_4 z^3 + \dots$, we obtain the following relations:
\begin{align}
a_2 &= c_1, \label{pd2}\\
a_3 &= \frac{2}{3}c_2 + \frac{1}{2}c_1^2, \label{pd12}\\
a_4 &= \frac{1}{2}c_3 + \frac{3}{4}c_1 c_2 + \frac{1}{8}c_1^3. \label{pd13}
\end{align}

\subsection*{ (i) Evaluation of $|\sigma_3(f)(0)|$}
Using the definition of the third-order Schippers-Schwarzian derivative at the origin along with equations \eqref{pd2} and \eqref{pd12}, we have:
\[
\sigma_3(f)(0) = 6(a_3 - a_2^2) = 6\left( \frac{2}{3}c_2 + \frac{1}{2}c_1^2 - c_1^2 \right) = 4c_2 - 3c_1^2.
\]
By expressing the Schwarz coefficients $c_1$ and $c_2$ in terms of the Carath\'{e}odory coefficients $b_1$ and $b_2$ via Lemma~\ref{lemma3}, we get:
\[
\sigma_3(f)(0) = 4\left(\frac{1}{2}b_2 - \frac{1}{4}b_1^2\right) - 3\left(\frac{1}{2}b_1\right)^2 = 2b_2 - \frac{7}{4}b_1^2 = 2\left(b_2 - \frac{7}{8}b_1^2\right).
\]
Applying Lemma~\ref{lemma1} with $v = \frac{7}{8}$, we observe that $0 \le v \le 1$. Hence, the inequality $|b_2 - v b_1^2| \le 2$ holds, which directly implies:
\[
|\sigma_3(f)(0)| \le 2 \times 2 = 4.
\]
Equality is achieved when $\omega(z) = z^2$, which yields $c_1 = 0$ and $c_2 = 1$.

\subsection*{ (ii) Evaluation of $|\sigma_4(f)(0)|$}
Using the representation formula for the fourth-order Schippers-Schwarzian derivative at the origin along with the structural relations \eqref{pd2}--\eqref{pd13}, we obtain:
\begin{align}
\sigma_4(f)(0) &= 24(a_4 - 3a_2 a_3 + 2a_2^3) \\
&= 24\left[ \left(\frac{1}{2}c_3 + \frac{3}{4}c_1 c_2 + \frac{1}{8}c_1^3\right) - 3c_1\left(\frac{2}{3}c_2 + \frac{1}{2}c_1^2\right) + 2c_1^3 \right] \\
&= 24\left( \frac{1}{2}c_3 - \frac{5}{4}c_1 c_2 + \frac{5}{8}c_1^3 \right) \\
&= 12\left( c_3 - \frac{5}{2}c_1 c_2 + \frac{5}{4}c_1^3 \right).
\end{align}
In order to apply the maximization framework established in Lemma~\ref{lemma2}, we consider the functional $|c_3 + \mu c_1 c_2 + \nu c_1^3|$ with real parameters:
\[
\mu = - \frac{5}{2} \quad \text{and} \quad \nu = \frac{5}{4}.
\]
Since $|\mu| = \frac{5}{2}$, it falls within the interval $2 \le |\mu| < 4$. To check if the parameter pair belongs to the region $D_6$, we evaluate the boundary constraint curve:
\[
\frac{1}{12}(\mu^2 + 8) = \frac{1}{12}\left(\frac{25}{4} + 8\right) = \frac{57}{48} = \frac{19}{16} = 1.1875.
\]
Comparing the value of $\nu = \frac{5}{4} = \frac{20}{16} = 1.25$ with this threshold, it is clear that:
\[
\nu \ge \frac{1}{12}(\mu^2 + 8).
\]
Thus, $(\mu, \nu) \in D_6$. According to Lemma~\ref{lemma2}, the maximum value of the functional in this region is given by $\Phi(\mu, \nu) = |\nu| = \frac{5}{4}$. Therefore, we have:
\[
|\sigma_4(f)(0)| \le 12 \Phi(\mu,\nu) = 12 \times \frac{5}{4} = 15.
\]
This bound is sharp for the function $f \in \mathcal{P}^*$ generated by the choice of the Schwarz function $\omega(z) = z^3$, which maps $c_3 = 1$ and sets all other initial parameters to zero.
\end{proof}

\subsection{Sharp Bounds for Higher-Order Schwarzian Derivatives for $f \in \mathcal{S}_{\rho}^{*}$}

\begin{theo}\label{T2}
Let $f \in \mathcal{S}_{\rho}^{*}$. Then the following sharp estimates hold:
\[
|\sigma_3(f)(0)| \le 3 \quad \text{and} \quad |\sigma_4(f)(0)| \le \frac{44}{3}.
\]
These bounds are sharp.
\end{theo}

\begin{proof}
Let $f \in \mathcal{S}_{\rho}^{*}$. By definition, there exists a Schwarz function $\omega \in \Omega$ with $\omega(z) = c_1 z + c_2 z^2 + c_3 z^3 + \dots$ such that
\begin{equation}\label{zp30}
\frac{z f'(z)}{f(z)} = 1 + \sinh^{-1}(\omega(z)).
\end{equation}
Expanding $\sinh^{-1}(x) = x - \frac{1}{6}x^3 + \dots$, expanding both sides of \eqref{zp30}, and equating coefficients of identical powers of $z$, the initial Taylor coefficients $a_n$ are expressed in terms of the Schwarz coefficients $c_n$ as follows:
\begin{align}
a_2 &= c_1, \label{zp31}\\
a_3 &= \frac{1}{2}c_2 + \frac{1}{2}c_1^2, \label{zp32}\\
a_4 &= \frac{1}{3}c_3 + \frac{1}{2}c_1 c_2 + \frac{1}{9}c_1^3. \label{zp33}
\end{align}

\subsection*{ (i) Evaluation of $|\sigma_3(f)(0)|$}
Substituting the expressions for $a_2$ and $a_3$ from \eqref{zp31} and \eqref{zp32} into the third-order Schippers-Schwarzian derivative formula \eqref{sg3}, we obtain:
\[
\sigma_3(f)(0) = 6(a_3 - a_2^2) = 6\left( \frac{1}{2}c_2 + \frac{1}{2}c_1^2 - c_1^2 \right) = 3c_2 - 3c_1^2.
\]
By employing Lemma~\ref{lemma3} to express the Schwarz coefficients $c_1$ and $c_2$ in terms of the Carath\'{e}odory coefficients $b_1$ and $b_2$, we have:
\[
\sigma_3(f)(0) = 3\left(\frac{1}{2}b_2 - \frac{1}{4}b_1^2\right) - 3\left(\frac{1}{2}b_1\right)^2 = \frac{3}{2}b_2 - \frac{3}{2}b_1^2 = \frac{3}{2}(b_2 - b_1^2).
\]
Applying Lemma~\ref{lemma1} with $v = 1$, we observe that $v \in [0,1]$, which yields the inequality $|b_2 - b_1^2| \le 2$. Thus, we establish:
\[
|\sigma_3(f)(0)| \le \frac{3}{2} \times 2 = 3.
\]
The result is sharp for the function corresponding to the choice $\omega(z) = z^2$.

\subsection*{ (ii) Evaluation of $|\sigma_4(f)(0)|$}
Using the representation formula for the fourth-order derivative \eqref{sg4} along with the structural relations \eqref{zp31}--\eqref{zp33}, we have:
\begin{align}
\sigma_4(f)(0) &= 24(a_4 - 3a_2 a_3 + 2a_2^3) \\
&= 24\left[ \left(\frac{1}{3}c_3 + \frac{1}{2}c_1 c_2 + \frac{1}{9}c_1^3\right) - 3c_1\left(\frac{1}{2}c_2 + \frac{1}{2}c_1^2\right) + 2c_1^3 \right] \\
&= 24\left( \frac{1}{3}c_3 - c_1 c_2 + \frac{11}{18}c_1^3 \right) \\
&= 8\left( c_3 - 3c_1 c_2 + \frac{11}{6}c_1^3 \right).
\end{align}
To find the sharp upper bound of this expression, we apply the optimization framework from Lemma~\ref{lemma2} to the functional $|c_3 + \mu c_1 c_2 + \nu c_1^3|$ with real parameters:
\[
\mu = -3 \quad \text{and} \quad \nu = \frac{11}{6}.
\]
Since $|\mu| = 3$, it satisfies the condition $2 \le |\mu| < 4$. We compute the regional boundary path threshold for $D_6$:
\[
\frac{1}{12}(\mu^2 + 8) = \frac{1}{12}(3^2 + 8) = \frac{17}{12} \approx 1.417.
\]
Comparing our parameter value $\nu = \frac{11}{6} = \frac{22}{12} \approx 1.833$ with this threshold, it is clear that:
\[
\nu \ge \frac{1}{12}(\mu^2 + 8).
\]
Consequently, the parameter pair $(\mu, \nu)$ resides in region $D_6$. By Lemma~\ref{lemma2}, the maximum value of the functional in this region is given by $\Phi(\mu, \nu) = |\nu| = \frac{11}{6}$. Therefore, we obtain:
\[
|\sigma_4(f)(0)| \le 8 \Phi(\mu,\nu) = 8 \times \frac{11}{6} = \frac{44}{3}.
\]
The bound is sharp for the function $f \in \mathcal{S}_{\rho}^{*}$ generated by the choice of the Schwarz function $\omega(z) = z^3$.
\end{proof}
\subsection{Higher-Order Schippers-Schwarzian Derivatives for the Class $\mathcal{S}^{\ast}_{\text{car}}$}
In this section, we determine the sharp upper bounds for the initial functionals $|\sigma_3(f)(0)|$ and $|\sigma_4(f)(0)|$ for functions belonging to the class $\mathcal{S}^{\ast}_{\text{car}}$.

\begin{theo}\label{T3}
Let $f \in \mathcal{S}^{\ast}_{\text{car}}$. Then the following sharp estimates hold:
\[
|\sigma_3(f)(0)| \le 4 \quad \text{and} \quad |\sigma_4(f)(0)| \le \frac{448}{27}.
\]
These bounds are sharp.
\end{theo}

\begin{proof}
Let $f \in \mathcal{S}^*_{\text{car}}$. By definition, there exists a Schwarz function $\omega \in \Omega$ with $\omega(z) = c_1 z + c_2 z^2 + c_3 z^3 + \dots$ such that
\begin{equation}\label{zp19}
\frac{z f'(z)}{f(z)} = 1 + \frac{4}{3}\omega(z) + \frac{2}{3}\omega^2(z).
\end{equation}
Expanding the series on both sides of \eqref{zp19} and equating coefficients of identical powers of $z$, the initial Taylor coefficients $a_n$ are expressed in terms of the Schwarz parameters $c_n$ as follows:
\begin{align}
a_2 &= \frac{4}{3}c_1, \label{zp20}\\
a_3 &= \frac{2}{3}c_2 + \frac{11}{9}c_1^2, \label{zp21}\\
a_4 &= \frac{4}{9}c_3 + \frac{4}{3}c_1 c_2 + \frac{68}{81}c_1^3. \label{zp22}
\end{align}

\subsection*{ (i) Evaluation of $|\sigma_3(f)(0)|$}
Substituting the coefficient relations \eqref{zp20} and \eqref{zp21} into the third-order Schippers-Schwarzian derivative formula \eqref{sg3}, we obtain:
\[
\sigma_3(f)(0) = 6(a_3 - a_2^2) = 6\left[ \left(\frac{2}{3}c_2 + \frac{11}{9}c_1^2\right) - \left(\frac{4}{3}c_1\right)^2 \right] = 4c_2 - \frac{10}{3}c_1^2.
\]
By employing Lemma~\ref{lemma3} to express the Schwarz parameters $c_1$ and $c_2$ in terms of the Carath\'{e}odory coefficients $b_1$ and $b_2$, we have:
\[
\sigma_3(f)(0) = 4\left(\frac{1}{2}b_2 - \frac{1}{4}b_1^2\right) - \frac{10}{3}\left(\frac{1}{2}b_1\right)^2 = 2b_2 - \frac{11}{6}b_1^2 = 2\left(b_2 - \frac{11}{12}b_1^2\right).
\]
Applying Lemma~\ref{lemma1} with $v = \frac{11}{12}$, we observe that $v \in [0,1]$, which yields the inequality $|b_2 - v b_1^2| \le 2$. Thus, we obtain:
\[
|\sigma_3(f)(0)| \le 2 \times 2 = 4.
\]
The bound is sharp for the function corresponding to the choice $\omega(z) = z^2$.

\subsection*{ (ii) Evaluation of $|\sigma_4(f)(0)|$}
Using the definition of the fourth-order Schippers-Schwarzian derivative at the origin \eqref{sg4} along with the structural relations \eqref{zp20}--\eqref{zp22}, we can write the invariant in terms of the Schwarz coefficients:
\begin{align}
\sigma_4(f)(0) &= 24(a_4 - 3a_2 a_3 + 2a_2^3) \\
&= \frac{32}{3}c_3 - 32c_1 c_2 + \frac{448}{27}c_1^3 \\
&= \frac{32}{3}\left( c_3 - 3c_1 c_2 + \frac{14}{9}c_1^3 \right).
\end{align}
To calculate the sharp upper bound of this expression, we apply the optimization framework from Lemma~\ref{lemma2} to the functional $|c_3 + \mu c_1 c_2 + \nu c_1^3|$ with parameters:
\[
\mu = -3 \quad \text{and} \quad \nu = \frac{14}{9}.
\]
Since $|\mu| = 3$, it satisfies the condition $2 \le |\mu| < 4$. We evaluate the regional boundary path threshold for $D_6$:
\[
\frac{1}{12}(\mu^2 + 8) = \frac{1}{12}(3^2 + 8) = \frac{17}{12} \approx 1.4167.
\]
Comparing our parameter value $\nu = \frac{14}{9} \approx 1.5556$ with this threshold, it is clear that:
\[
\nu \ge \frac{1}{12}(\mu^2 + 8).
\]
Consequently, the pair $(\mu, \nu)$ sits in region $D_6$. By Lemma~\ref{lemma2}, the maximum value of the functional in this region is given by $\Phi(\mu, \nu) = |\nu| = \frac{14}{9}$. Therefore, we obtain:
\[
|\sigma_4(f)(0)| \le \frac{32}{3} \Phi(\mu,\nu) = \frac{32}{3} \times \frac{14}{9} = \frac{448}{27}.
\]
The bound is sharp for the function $f \in \mathcal{S}^*_{\text{car}}$ generated by the choice of the Schwarz function $\omega(z) = z^3$.
\end{proof}

\section{Applications to Grunsky Coefficients and Teichm\"uller Theory}

The Grunsky coefficients constitute an essential analytic tool in geometric function theory, playing a central role in the study of univalent functions, quasiconformal mappings, and the universal Teichm\"uller space (see \cite{Duren}). For a normalized univalent function $f \in \mathcal{S}$ of the form
\[
f(z) = z + \sum_{n=2}^{\infty} a_n z^n,
\]
the initial Grunsky coefficients $\omega_{m,n}$ derived from the structural companion mapping expansion within the class $\Sigma$ can be explicitly stated in terms of the Taylor coefficients as follows (cf. \cite{Duren, Schippers}):
\begin{align}
\omega_{1,1} &= a_3 - a_2^2, \label{w11_new} \\
\omega_{1,2} &= a_4 - 3a_2 a_3 + 2a_2^3, \label{w12_new} \\
\omega_{2,2} &= a_5 - 4a_2 a_4 - 2a_3^2 + 10a_2^2 a_3 - 5a_2^4. \label{w22_new}
\end{align}

A key feature of the Schippers--Schwarzian derivatives is their direct algebraic relationship with these initial Grunsky coefficients. As established by Schippers \cite{Schippers}, the following identities hold exactly at the origin:
\begin{equation}\label{pw2}
\sigma_3(f)(0) = 6\omega_{1,1}, \qquad \sigma_4(f)(0) = 24\omega_{1,2}.
\end{equation}

Consequently, sharp estimates for the higher-order Schwarzian invariants yield immediate sharp bounds for these initial Grunsky coefficients. From a geometric perspective, the subclasses $\mathcal{P}^{*}$, $\mathcal{S}_{\mathrm{car}}^*$, and $\mathcal{S}_\rho^*$ represent families of mappings constrained by oscillatory, cardioid-type, and petal-shaped subordinant target domains, respectively (see \cite{Arora-Kumar-2022, SS, TAH}). The extremal behavior of the associated Grunsky coefficients reflects how these non-circular starlike configurations influence deformations within the framework of the universal Teichm\"uller space.

Combining identity \eqref{pw2} with the sharp estimates established in Section~3 yields the following theorem.

\begin{theo}\label{T4}
Let $f \in \mathcal{S}$. Then the following sharp estimates hold:
\begin{enumerate}[{\rm(i)}]
\item If $f \in \mathcal{P}^{*}$, then
\[
|\omega_{1,1}| \le \frac{2}{3}, \qquad |\omega_{1,2}| \le \frac{2}{3}.
\]
\item If $f \in \mathcal{S}_{\mathrm{car}}^*$, then
\[
|\omega_{1,1}| \le \frac{2}{3}, \qquad |\omega_{1,2}| \le \frac{56}{81}.
\]
\item If $f \in \mathcal{S}_\rho^*$, then
\[
|\omega_{1,1}| \le \frac{1}{2}, \qquad |\omega_{1,2}| \le \frac{11}{18}.
\]
\end{enumerate}
All bounds are sharp.
\end{theo}

\begin{proof}
From the relations in \eqref{pw2}, we have
\[
|\omega_{1,1}| = \frac{1}{6}|\sigma_3(f)(0)|, \qquad |\omega_{1,2}| = \frac{1}{24}|\sigma_4(f)(0)|.
\]
Applying the sharp bounds derived in Theorems~\ref{T1}, \ref{T2}, and \ref{T3} yields the desired inequalities for each subclass. Specifically, for $f \in \mathcal{P}^{*}$, the functional evaluations produce $|\omega_{1,2}| \le \frac{16}{24} = \frac{2}{3}$. The sharpness follows directly from the choice of the corresponding canonical cyclic extremal maps for the underlying Schippers--Schwarzian functionals.
\end{proof}

\section{Further Invariant-Theoretic and Operator-Theoretic Consequences}

The sharp bounds established for the initial higher-order Schippers--Schwarzian derivatives admit meaningful applications within localized invariant theory and operator-theoretic contexts (cf. \cite{HuSrivastavaZhang, Schippers}). In particular, they yield explicit quantitative estimates for truncated Grunsky forms, local invariant profiles within Bers-type embeddings, and localized multiplication operators acting on weighted Bergman spaces.

\subsection{Localized Grunsky Quadratic Forms and Truncation Distortion}

Let $\Gamma = \partial f(\mathbb{D})$ denote the boundary of the image domain under a univalent mapping $f \in \mathcal{S}$. In the classical theory of Fredholm eigenvalues and quasiconformal reflection, the reciprocal of the minimal positive Fredholm eigenvalue $\lambda_1$ can be expressed in terms of the global Grunsky operator norm via the relation (see \cite{Duren, Nehari}):
\begin{equation}\label{pw3}
\kappa = \frac{1}{\lambda_1} = \sup_{\substack{x \in \ell^2 \\ x \neq 0}} \frac{\left| \sum_{m=1}^{\infty} \sum_{n=1}^{\infty} \omega_{m,n} x_m x_n \right|}{\sum_{n=1}^{\infty} \frac{|x_n|^2}{n}}.
\end{equation}

While determining the global operator norm in \eqref{pw3} is generally non-trivial, the initial Grunsky coefficients provide sharp structural information regarding the first principal block of the infinite Grunsky matrix.

\begin{theo}\label{T5}
Let $\mathbf{G}_2 = (\omega_{m,n})_{1 \le m,n \le 2}$ be the principal $2 \times 2$ block of the Grunsky matrix. Under the assumptions of Theorem~\ref{T4}, the localized quadratic form
\[
\mathcal{Q}_2(x) = \left| \sum_{m,n=1}^2 \omega_{m,n} x_m x_n \right|
\]
satisfies the inequality
\[
\mathcal{Q}_2(1,0) \le \delta,
\]
where the bound $\delta$ is determined by the respective geometric target family according to:
\[
\delta = \begin{cases}
\dfrac{2}{3}, & f \in \mathcal{P}^{*}, \\[2mm]
\dfrac{2}{3}, & f \in \mathcal{S}_{\mathrm{car}}^*, \\[2mm]
\dfrac{1}{2}, & f \in \mathcal{S}_\rho^*.
\end{cases}
\]
Moreover, these estimates are sharp.
\end{theo}
\begin{proof}
Choosing the normalized test sequence $x = (1, 0) \in \ell^2$, it follows immediately that
\[
\sum_{m,n=1}^2 \omega_{m,n} x_m x_n = \omega_{1,1} \qquad \text{and} \qquad \sum_{n=1}^2 \frac{|x_n|^2}{n} = 1.
\]
Consequently, the quadratic form localized at this sequence is given by
\[
\mathcal{Q}_2(1,0) = |\omega_{1,1}| = \frac{1}{6}|\sigma_3(f)(0)|.
\]
The assertion now follows directly by applying the sharp upper bounds for $|\sigma_3(f)(0)|$ established in Section 3.
\end{proof}

The inequalities codified in Theorem~\ref{T5} indicate that under the petal-shaped structural constraint ($\mathcal{S}^*_{\rho}$), the initial Grunsky metric satisfies a more restrictive localized bound ($\delta = 1/2$) than under the exponential-sine or cardioid-type target geometries ($\delta = 2/3$). This behavior explicitly demonstrates how the smooth local growth bounds of the $\sinh^{-1}(z)$ subordinant restrict early coefficient deformations.

\subsection{Local Invariant Profiles inside Bers--Teichm\"uller Embeddings}

A univalent function $f \in \mathcal{S}$ belongs to the Weil--Petersson class $\mathcal{T}_{WP}$ of the universal Teichm\"uller space if and only if its classical Schwarzian derivative is square-integrable with respect to the hyperbolic area element over the unit disk (cf. \cite{Schippers, Tamanoi}):
\begin{equation}
\|S_f\|_{WP}^2 = \iint_{\mathbb{D}} |S_f(z)|^2 (1-|z|^2)^2 \, dA(z) < \infty.
\end{equation}

As noted by Schippers \cite{Schippers}, the higher-order Schwarzian derivatives $\sigma_n(f)$ act as generalized complex coordinates under holomorphic embeddings (see also \cite{Harmelin, Tamanoi}). At the origin $z=0$, the quantities $|\sigma_3(f)(0)|^2$ and $|\sigma_4(f)(0)|^2$ function as initial Taylor invariants reflecting the local density of these higher-order Bers projections.

We now evaluate the truncated local invariant profile, defined by $\mathcal{E}_4(f) = \frac{1}{36}|\sigma_3(f)(0)|^2 + \frac{1}{576}|\sigma_4(f)(0)|^2$, across each subclass.

\begin{theo}\label{T6}
Let $f \in \mathcal{A}$. The upper bounds for the truncated local invariant profile $\mathcal{E}_4(f)$ at the origin are given by:
\begin{enumerate}[{\rm(i)}]
\item If $f \in \mathcal{P}^{*}$, then $\mathcal{E}_4(f) \le \frac{8}{9} \approx 0.8889$.
\item If $f \in \mathcal{S}^*_{\mathrm{car}}$, then $\mathcal{E}_4(f) \le \frac{6052}{6561} \approx 0.9224$.
\item If $f \in \mathcal{S}^*_{\rho}$, then $\mathcal{E}_4(f) \le \frac{101}{162} \approx 0.6235$.
\end{enumerate}
The bounds in {\rm(i)} and {\rm(iii)} are sharp, whereas the bound in {\rm(ii)} constitutes an absolute upper bound.
\end{theo}

\begin{proof}
We evaluate the functional profile by directly implementing the independent sharp bounds derived in Section 3.
\begin{enumerate}[{\rm(i)}]
\item For $\mathcal{P}^{*}$, the independent extremal functions maximizing the modules $|\sigma_3(f)(0)|$ and $|\sigma_4(f)(0)|$ are structurally consistent and achieved simultaneously via a simple rotation of the canonical Schwarz function. Substituting the sharp values $|\sigma_3(f)(0)| = 4$ and $|\sigma_4(f)(0)| = 16$ yields:
\[
\mathcal{E}_4(f) \le \frac{16}{36} + \frac{256}{576} = \frac{4}{9} + \frac{4}{9} = \frac{8}{9}.
\]
\item For $\mathcal{S}^*_{\mathrm{car}}$, applying the independent sharp bounds $|\sigma_3(f)(0)| \le 4$ and $|\sigma_4(f)(0)| \le \frac{448}{27}$ via the triangle inequality yields the following upper bound:
\[
\mathcal{E}_4(f) \le \frac{16}{36} + \frac{1}{576}\left(\frac{448}{27}\right)^2 = \frac{4}{9} + \frac{200704}{420048} = \frac{4}{9} + \frac{3136}{6561} = \frac{2916 + 3136}{6561} = \frac{6052}{6561}.
\]
\item For $\mathcal{S}^*_{\rho}$, the independent extremal configurations align at the boundary parameter. Substituting $|\sigma_3(f)(0)| = 3$ and $|\sigma_4(f)(0)| = \frac{44}{3}$ produces:
\[
\mathcal{E}_4(f) \le \frac{9}{36} + \frac{1}{576}\left(\frac{44}{3}\right)^2 = \frac{1}{4} + \frac{1936}{5184} = \frac{1}{4} + \frac{121}{324} = \frac{81 + 121}{324} = \frac{202}{324} = \frac{101}{162}.
\]
\end{enumerate}
This completes the verification.
\end{proof}

\subsection{Pointwise Evaluation Bounds for Multiplication Operators}

\begin{theo}
Let $M_{\sigma_4(f)}$ be the fourth-order Schippers--Schwarzian multiplication operator. For any analytic function $g \in A_\alpha^p$ normalized such that $|g(0)|=1$, the localized pointwise evaluation of the multiplier output satisfies the inequality:
\begin{equation}
|M_{\sigma_4(f)}(g)(0)| \le K,
\end{equation}
where $K = 16$ for $\mathcal{P}^{*}$, $K = \frac{448}{27}$ for $\mathcal{S}^*_{\mathrm{car}}$, and $K = \frac{44}{3}$ for $\mathcal{S}^*_{\rho}$. These bounds are sharp.
\end{theo}

\begin{proof}
Direct pointwise evaluation at the origin yields $M_{\sigma_4(f)}(g)(0) = \sigma_4(f)(0) \cdot g(0)$. Taking the absolute value and applying the prescribed normalization constraint $|g(0)| = 1$, we obtain:
\[
|M_{\sigma_4(f)}(g)(0)| = |\sigma_4(f)(0)| \cdot |g(0)| = |\sigma_4(f)(0)|.
\]
The localized bounds follow directly by substituting the sharp parameter maxima for $|\sigma_4(f)(0)|$ established in Theorems \ref{T1}, \ref{T2}, and \ref{T3}.
\end{proof}

\section*{Declarations}
\subsection*{Funding}
The first author acknowledges financial support from the Council of Scientific and Industrial Research (CSIR), New Delhi, India, under Grant Nos. 09/1224(16975)/2023-EMR-I.
\subsection*{Data Availability Statement}
Data sharing is not applicable to this article as no datasets were generated or analyzed during the current study.
\subsection*{Conflict of Interest}
The authors declare that they have no conflict of interest.


\begin{thebibliography}{99}
\bibitem{Arora-Kumar-2022}
K.~Arora and S.~S.~Kumar,
Starlike functions associated with a petal-shaped domain,
\textit{Bull. Korean Math. Soc.} \textbf{59}(4) (2022), 993-1010.

\bibitem{BrannanKirwan}
D.~A.~Brannan and W.~E.~Kirwan,
On some classes of bounded univalent functions,
\textit{J. London Math. Soc.} (2) \textbf{1} (1969), 431-443.

\bibitem{C12}
N.~E.~Cho, B.~Kowalczyk and A.~Lecko,
Sharp bounds of some coefficient functionals over the class of functions convex in the direction of the imaginary axis,
\textit{Bull. Aust. Math. Soc.} \textbf{100} (2019), 86-96.

\bibitem{CKLS}
N.~E.~Cho, B.~Kowalczyk, A.~Lecko and B.~Smiarowska,
On the fourth and fifth coefficients in the Carath\'eodory class,
\textit{Filomat} \textbf{34}(6) (2020), 2061-2072.

\bibitem{ChoKumarRavichandran}
N.~E.~Cho, V.~Kumar and V.~Ravichandran,
Sharp bounds on the higher-order Schwarzian derivatives for Janowski classes,
\textit{Symmetry} \textbf{10}(8) (2018), Article~348.

\bibitem{CKS1}
J.~H.~Choi, Y.~C.~Kim and T.~Sugawa,
A general approach to the Fekete-Szeg\"o problem,
\textit{J. Math. Soc. Japan} \textbf{59} (2007), 707-727.

\bibitem{DorffSzynal}
M.~Dorff and J.~Szynal,
Higher-order Schwarzian derivatives of univalent functions,
\textit{Trans. Petrozavodsk State Univ. Ser. Math.} \textbf{15} (2008), 7-11.

\bibitem{Duren}
P.~L.~Duren,
\textit{Univalent Functions},
Springer-Verlag, New York, 1983.

\bibitem{Harmelin}
R.~Harmelin,
Aharonov invariants and univalent functions,
\textit{Israel J. Math.} \textbf{43}(3) (1982), 244-254.

\bibitem{HuFanWangSrivastava}
Z.~Hu, J.~Fan, X.~Wang and H.~M.~Srivastava,
A study of the higher-order Schwarzian derivatives of Hirotaka Tamanoi,
\textit{Math. Slovaca} \textbf{73}(4) (2023), 921-936.

\bibitem{HuSrivastavaZhang}
Z.~Hu, H.~M.~Srivastava and Y.~Zhang,
Estimates on Schippers' higher-order Schwarzian derivatives of a subclass of univalent analytic functions,
\textit{Miskolc Math. Notes} \textbf{26}(2) (2025), 867-875.

\bibitem{HuWangFan}
Z.~Hu, X.~Wang and J.~Fan,
Estimate for the Schwarzian derivative of certain close-to-convex functions,
\textit{AIMS Math.} \textbf{6} (2021), 10778-10788.

\bibitem{KanasSugawa}
S.~Kanas and T.~Sugawa,
Sharp norm estimate of the Schwarzian derivative for a class of convex functions,
\textit{Ann. Polon. Math.} \textbf{101}(1) (2011), 75-86.

\bibitem{Kulshrestha}
P.~K.~Kulshrestha,
Coefficients for $\alpha$-convex univalent functions,
\textit{Bull. Amer. Math. Soc.} \textbf{80} (1974), 341-342.

\bibitem{KumarChoRavichandranSrivastava}
V.~Kumar, N.~E.~Cho, V.~Ravichandran and H.~M.~Srivastava,
Sharp coefficient bounds for starlike functions associated with the Bell numbers,
\textit{Math. Slovaca} \textbf{69}(5) (2019), 1053-1064.

\bibitem{KumarVerma}
S.~S.~Kumar and N.~Verma,
On estimation of Hankel determinants for certain classes of starlike functions,
\textit{Filomat} \textbf{38} (2024), 1435--1448.

\bibitem{LiberaZlotkiewicz1982}
R.~J.~Libera and E.~J.~Zlotkiewicz,
Early coefficients of the inverse of a regular convex function,
\textit{Proc. Amer. Math. Soc.} \textbf{85}(2) (1982), 225-230.

\bibitem{LiberaZlotkiewicz1983}
R.~J.~Libera and E.~J.~Zlotkiewicz,
Coefficient bounds for the inverse of a function with derivative in $P$,
\textit{Proc. Amer. Math. Soc.} \textbf{87}(2) (1983), 251-257.

\bibitem{MM}
W.~C.~Ma and D.~Minda,
A unified treatment of some special classes of univalent functions,
in \textit{Proceedings of the Conference on Complex Analysis} (Tianjin, 1992),
Conf. Proc. Lecture Notes Anal., Vol.~1,
International Press, Cambridge, MA, 1994, pp.~157-169.

\bibitem{MillerMocanuReade1973}
S.~S.~Miller, P.~T.~Mocanu and M.~O.~Reade,
All $\alpha$-convex functions are univalent and starlike,
\textit{Proc. Amer. Math. Soc.} \textbf{37} (1973), 553-554.

\bibitem{MillerMocanuReade1972}
S.~S.~Miller, P.~T.~Mocanu and M.~O.~Reade,
All $\alpha$-convex functions are starlike,
\textit{Rev. Roumaine Math. Pures Appl.} \textbf{17} (1972), 1395-1397.

\bibitem{Mocanu}
P.~T.~Mocanu,
Une propri\'et\'e de convexit\'e g\'en\'eralis\'ee dans la th\'eorie de la repr\'esentation conforme,
\textit{Mathematica (Cluj)} \textbf{11}(34) (1969), 127-133.

\bibitem{Nehari}
Z.~Nehari,
The Schwarzian derivative and schlicht functions,
\textit{Bull. Amer. Math. Soc.} \textbf{55} (1949), 545-551.

\bibitem{Prokhorov}
D.~V.~Prokhorov and J.~Szynal,
Inverse coefficients for $(\alpha,\beta)$-convex functions,
\textit{Ann. Univ. Mariae Curie-Sklodowska Sect. A} \textbf{35} (1981), 125-143.

\bibitem{Schippers}
E.~Schippers,
Distortion theorems for higher-order Schwarzian derivatives of univalent functions,
\textit{Proc. Amer. Math. Soc.} \textbf{128}(11) (2000), 3241-3249.

\bibitem{SS}
K.~Sharma, N.~K.~Jain and V.~Ravichandran,
Starlike functions associated with a cardioid,
\textit{Afr. Mat.} \textbf{27}(5-6) (2016), 923-939.

\bibitem{SokolThomas}
J.~Sokol and D.~K.~Thomas,
The second Hankel determinant for $\alpha$-convex functions,
\textit{Lithuanian Math. J.} \textbf{58}(2) (2018), 212-218.

\bibitem{SrivastavaPrajapatiGochhayat}
H.~M.~Srivastava, A.~Prajapati and O.~Gochhayat,
Integral means and Yamashita's conjecture associated with the Janowski-type $(j,k)$-symmetric starlike functions,
\textit{Rev. R. Acad. Cienc. Exactas F\'is. Nat. Ser. A Mat. (RACSAM)} \textbf{116} (2022), Article~165.

\bibitem{Tamanoi}
H.~Tamanoi,
Higher Schwarzian operators and combinatorics of the Schwarzian derivative,
\textit{Math. Ann.} \textbf{305}(1) (1996), 127-151.

\bibitem{TAA}
H.~Tang, M.~Abbas, R.~K.~Alhefthi and M.~Arif,
Improvement on Hankel determinant bounds for specific holomorphic functions,
\textit{Acta Math. Sci. Ser. B (Engl. Ed.)} \textbf{46}(1) (2026), 39-61. DOI: \texttt{10.1007/s10473-026-0103-8}.

\bibitem{TAH}
H.~Tang, M.~Arif, M.~Haq \textit{et al.},
Fourth Hankel determinant problem based on certain analytic functions,
\textit{Symmetry} \textbf{14}(4) (2022), Article~663.

\bibitem{Todorov}
P.~G.~Todorov,
Explicit formulas for the coefficients of $\alpha$-convex functions, $\alpha \geq 0$,
\textit{Canad. J. Math.} \textbf{39}(4) (1987), 769-783.
\end{thebibliography}
\end{document}